\documentclass[a4paper]{article}
\usepackage{amsmath, amsfonts}
\usepackage{amssymb, latexsym}
\usepackage{amsthm}
\usepackage{mathtools}
\usepackage{empheq}
\usepackage{xypic}
\usepackage{stmaryrd}
\allowdisplaybreaks
\usepackage{easy-todo}

\usepackage{authblk}

\makeatletter
\newcommand{\leqnomode}{\tagsleft@true\let\veqno\@@leqno}
\newcommand{\reqnomode}{\tagsleft@false\let\veqno\@@eqno}
\makeatother

\usepackage[
colorlinks=true,       
linkcolor=blue,          
citecolor=blue,        
plainpages=false,      
pdfpagelabels,
]{hyperref}

\usepackage[shortlabels]{enumitem}
\setlist[enumerate]{leftmargin=*,align=left,labelindent=\parindent}

\newcommand{\gen}{\textup{\texttt{gen}}}

\newcommand{\Binary}{\mathsf{Bin}}

\DeclareMathOperator{\imp}{\,\rightarrow\,}

\newcommand{\dotminus}{\mathbin{\ooalign{\hss\raise.6ex\hbox{$\cdot$}\hss\crcr$-$}}}

\DeclareMathOperator{\Ind}{\textrm{\textit{Ind}}}

\newcommand{\Con}{\mathsf{Con}}

\newcommand{\Base}{\mathbb{N}}
\newcommand{\calBR}{\mathcal{BR}}

\newcommand{\nil}{\langle\, \rangle}
\newcommand{\seq}[1]{\langle#1\rangle}

\newcommand{\lh}[1]{\lvert #1 \rvert}
\newcommand{\CC}{\mathrm{\textup{AC}_{0}}}
\newcommand{\BR}{{\mathsf{BR}}}
\DeclareMathOperator{\sg}{\mathsf{sg}}

\newcommand{\AppendZero}[1]{\widehat{#1}}

\newcommand{\CZF}{\mathrm{CZF}}
\newcommand{\HAw}{\mathrm{HA}^{\omega}}

\newcommand{\SystemT}{\mathsf{T}}

\newcommand{\DT}{\mathrm{DT}}
\newcommand{\BT}{\mathrm{BT}}
\newcommand{\Rec}{\mathsf{Rec}}
\newcommand{\Succ}{\mathsf{Succ}}
\newcommand{\Zero}{0}

\newcommand{\Baire}{\Base^{\Base}}
\newcommand{\Cantor}{2^{\Base}}
\newcommand{\At}{\mathsf{At}}
\newcommand{\KE}{\mathsf{KE}}

\newcommand{\lift}[1]{{#1}^{\dagger}}

\newtheorem{theorem}{Theorem}[section]
\newtheorem{proposition}[theorem]{Proposition}
\newtheorem{lemma}[theorem]{Lemma}

\theoremstyle{definition}
\newtheorem{definition}[theorem]{Definition}
\theoremstyle{remark}
\newtheorem{remark}[theorem]{Remark}
\newtheorem*{notations*}{Notations}

\numberwithin{equation}{section}

\title{Representing definable functions of $\HAw$ \\
  by neighbourhood functions}

\author{Tatsuji Kawai}
\affil[]{Japan Advanced Institute of Science and Technology\authorcr
1-1 Asahidai, Nomi, Ishikawa 923-1292, Japan\authorcr
\texttt{tatsuji.kawai@jaist.ac.jp}}

\date{}

\begin{document}
\maketitle

\begin{abstract}
Brouwer (1927) 
claimed that every function from the Baire space to natural
numbers is induced by a neighbourhood function whose domain admits
bar induction. We show that Brouwer's claim is provable in Heyting
arithmetic in all finite types ($\HAw$) for definable functions of
the system. The proof does not rely on elaborate proof theoretic
methods such as normalisation or ordinal analysis. Instead, 
we internalise in $\HAw$ the dialogue tree interpretation of G\"odel's system~$\SystemT$
due to Escard\'o (2013). The interpretation determines 
a syntactic translation of terms,
which yields a neighbourhood function from a closed term of $\HAw$
with the required property. As applications of this result, we prove
some well-known properties of $\HAw$: uniform continuity of definable
functions from $\Base^{\Base}$ to $\Base$ on the Cantor space;
closure under the
rule of bar induction; and closure of bar recursion for the
lowest type with a definable stopping function.
\medskip

\end{abstract}
\section{Introduction}\label{sec:Introduciton}
In ``On the domains of definition of
functions''~\cite{BrouwerDomainsofFunctions}, Brouwer claims that
every function from the Baire space to natural numbers is not only
continuous but also contains a bar for which  so-called \emph{bar
induction} holds.  In terms of the modern constructive mathematics,
Brouwer's claim can be stated as follows, which we refer to as \emph{bar
theorem}:
\begin{quote}
For any function $f \colon \Base^{\Base} \to \Base$, there is a
neighbourhood function $\gamma \colon \Base^{*} \to \Base$ of $f$ such
that its domain $S_{\gamma} := \left\{ a \in \Base^{*} \mid \gamma(a)>
0 \right\}$ satisfies the following induction principle: any
inductive predicate $Q$ on ${\Base^{*}}$ which contains $S_{\gamma}$
necessarily  contains the empty sequence.
\end{quote}
Here, a predicate $Q$ on $\Base^{*}$ is inductive if $\forall
a^{\Base^{*}} \left[ \forall n^{\Base} Q(a * \langle n \rangle) \imp
Q(a) \right]$: if every one-step extension of $a$ satisfies $Q$, then
$a$ satisfies $Q$. A neighbourhood function of $f \colon \Baire \to \Base$ is
an algorithm which tells us whether a given initial segment of an
input $\alpha$ of $f$ is long enough to compute the value $f(\alpha)$;
when the initial segment is not long enough, its value stays at $0$
waiting for more input to be supplied; when it has read enough initial
segment of $\alpha$ to compute the value $f(\alpha)$, it outputs a
positive value $f(\alpha) + 1$.  See Section \ref{sec:Nhb} for the
precise definition.

The purpose of this paper is to show that bar theorem holds for closed
terms of Heyting arithmetic in all finite types $\HAw$ (i.e., closed terms
of G\"odel's system~$\SystemT$).  Specifically, given a closed term $Y
\colon \Baire \to \Base$ of $\HAw$, one can construct a neighbourhood
function of $Y$ as a closed term of $\HAw$ for which bar
induction is valid.  The existing literature suggests that our result is not
surprising: it is known that a closed term $Y \colon \Baire
\to \Base$ of system~$\SystemT$ has a
$\SystemT$-definable modulus of continuity (see e.g., Schwichtenberg
\cite{SchwichtenbergBarRec01}); moreover, $\HAw$ is closed under the
rule of bar induction (Howard \cite[Section 5]{HowardOrdinalAnalysisT}).%
\footnote{The proof by Howard in \cite[Section
5]{HowardOrdinalAnalysisT} applies to those variants of $\HAw$ that
admit G\"odel's Dialectica interpretation into system~$\SystemT$.
On the other hand, we work with the extensional version of $\HAw$,
which does not admit Dialectica interpretation (cf.\
Howard \cite{HowardMajorizable}).}
However, our proof does not rely on sophisticated proof theoretic
methods such as normalisation of infinite terms or ordinal analysis
used in those works. Nor do we use forcing, which is often used to
prove the fan rule, a weaker form of bar induction rule (see Beeson
\cite[Chapter XVI, Section 4]{BeesonFoundationConstMath}).

Instead, our proof of bar theorem is inspired by the dialogue tree model of
system~$\SystemT$ by Escard\'o \cite{EscardoEffectfulForcing} (see
Section \ref{sec:Dialogue}). His main idea is to represent a
$\SystemT$-definable function $f \colon \Baire \to \Base$ by a certain
well-founded tree, called dialogue tree, which can be thought of as a
computation tree of $f$. Since dialogue trees are inductively
defined, one can extract strong continuity properties of
$\SystemT$-definable functions. In this respect, his approach is
similar to the elimination of choice sequences (Kreisel and
Troelstra~\cite[Section 7]{KreiselTroelstra}), where a term containing
a variable for a choice sequence is represented as a Brouwer-operation
(see also Section \ref{sec:BrouwerTree}).
Our proof of bar theorem for $\HAw$ uses a mix of both approaches: the simplicity
of the dialogue model lends itself for direct formalisation in $\HAw$,
while the representation of terms by Brouwer-operations would
immediately yield a proof of bar theorem. We elaborate on how these
ideas can be combined to give a proof of bar theorem for closed terms
of $\HAw$.

The basic idea of our proof is to formalise Escard\'o's model in
$\HAw$. Instead of directly formalising his model, however, we extract
essential properties of the dialogue tree model that is needed for
the representation theorem. By so doing, we define a family
of models for the structure sharing these essential properties
(Section \ref{sec:NonStdRepresentation}). By instantiating this
abstract model with a structure other than dialogue trees, one obtains
a representation theorem of closed terms of $\HAw$ for that
particular structure. In particular, instantiated with
Brouwer-operations, the
model immediately yields a proof of bar theorem in $\HAw$ extended
with the type of Brouwer-operations (Section \ref{sec:BrouwerTree}).
Finally, the use of the transfinite type, that of Brouwer-operation, is
eliminated by reformulating the Brouwer-operation model in terms of
neighbourhood functions.

This last step of the proof, presented in Section
\ref{sec:ProofBarThm}, is inspired by Oliva and Steila
\cite{Oliva_Steila_bar_recursion_closure}, who showed that Spector's
bar recursion for the lowest type is definable in G\"odel's system~$\SystemT$ when its stopping function is $\SystemT$-definable. The
structure and technique used in their proof are similar to those of
ours. However, we believe that our proof is more perspicuous, having
presented its essential structure in a more abstract setting in
Section \ref{sec:NonStdRepresentation}. Moreover, as far as we know,
it is still open whether bar induction follows from bar recursion.  In
this respect, our result is stronger than their result (see Section
\ref{sec:BarRecursion}).

As applications of bar theorem, we prove some well-known properties of
$\HAw$: uniform continuity of definable functions from $\Base^{\Base}$
to $\Base$ on the Cantor space; closure under the rule of bar
induction; and closure of bar recursion for the lowest type with a
definable stopping function.

\begin{remark}
  Our proof of the existence of a $\SystemT$-definable neighbourhood
  function for a closed system-$\SystemT$ term is similar to the proof
  for the existence of a $\SystemT$-definable majorant for a closed
  system-$\SystemT$ term (see Howard \cite{HowardMajorizable} and
  Kohlenbach~\cite[Chapter 6]{KohlenbachAppliedProfTheory}).  As an
  application of existence of majorants,
  Kohlenbach~\cite{KohlenbachPWHereditaryMajorization} gave a simple
  proof of the fan rule for $\HAw$, which yields uniform continuity of
  $\SystemT$-definable functions from $\Base^{\Base}$ to $\Base$ on
  the Cantor space \cite[3.6 Application]{KohlenbachPWHereditaryMajorization}.
  His proof is much simpler than the one presented
  in Section \ref{sec:UCont}, although we derive the result from 
  a stronger result (i.e.\ the bar theorem).
\end{remark}

\paragraph{Organisation}
Section \ref{sec:HAw} fixes the formal system~$\HAw$;
Section~\ref{sec:Nhb} introduces bar theorem for $\HAw$;
Section~\ref{sec:NonStdRepresentation} formalises a family of models
abstracted from Escard\'o's dialogue tree model; Section~\ref{sec:ProofBarThm}
presents the proof of bar theorem for $\HAw$;
Section~\ref{sec:Application} presents applications of bar theorem.

%
%
The paper is essentially self-contained without
Section~\ref{sec:NonStdRepresentation}.  Thus, the reader who is only interested
in the proof of bar theorem and its applications can skip Section~\ref{sec:NonStdRepresentation}
entirely.  However, Section~\ref{sec:NonStdRepresentation} explains
how one can view our proof of bar theorem as an instance of
Escard\'o's dialogue tree model, thereby putting our work in a wider
picture.

\section{Heyting arithmetic in all finite types}\label{sec:HAw}
We work with the extensional version of Heyting arithmetic in all
finite types ($\HAw$) with lambda operators (see Troelstra~\cite[Section 1.8.4]{Troelstra1973}).

Finite types are defined as usual: $\Base$ is a type; if $\sigma,\tau$
are types, so is $\sigma \to \tau$, which is sometimes written
$\tau^{\sigma}$.  For convenience, we assume the existence of type
$\Base^{*}$ of finite sequences of objects of $\Base$, which is
identified with $\Base$ via coding.  We use metavariables
$\rho,\sigma,\tau$ for types.

Terms of $\HAw$ are those of simply typed lambda calculus with
natural number objects: There are demumerable list of variables
$x^{\rho}, y^{\rho}, z^{\rho},\ldots$ for
each type $\rho$, 
the lambda operator $\lambda x^{\rho}$,
and constants
$\Zero$,
$\Succ$, and 
$\Rec_{\rho}$ (for each type $\rho$) of the following types:
\begin{align*}
  \Zero &\colon \Base,
  &
  \Succ &\colon \Base \to \Base,
  &
  \Rec_{\rho} &\colon \rho \to (\Base \to \rho \to \rho) \to \Base \to \rho.
\end{align*}
A context is a finite list $x_{0}^{\rho_{0}}, \dots,
x_{n-1}^{\rho_{n-1}}$ of 
variables, which is sometimes written as $x_{0} \colon {\rho_{0}}, \dots,
x_{n-1} \colon {\rho_{n-1}}$. 
We use $\Gamma,\Delta$ for contexts.
Terms in contexts $\Gamma \vdash t \colon
\rho$ are inductively defined as follows:
\begin{equation}
  \label{def:TermInCont}
\begin{gathered}
  \Gamma, x^\rho, \Delta \vdash x \colon \rho
  \quad
  \Gamma \vdash \mathtt{C}\colon{\rho}
 \quad
  \frac{\Gamma, x^\rho \vdash t \colon \sigma}
  {\Gamma \vdash \lambda x^\rho. t \colon\rho \to \sigma}
  \quad 
  \frac {\Gamma \vdash u\colon{\rho \to \sigma}
  \quad 
  \Gamma \vdash v \colon {\rho}}
  {\Gamma \vdash uv \colon \sigma}
\end{gathered}
\end{equation}
where  $\mathtt{C}\colon {\rho}$ denotes a constant of type $\rho$.
The term $uv$ is sometimes written as $u(v)$.
Closed terms of $\HAw$ are terms in the empty context.

Prime formulas are equations $t =_{\rho} u$ between terms (in the same
context) of the same type $\rho$. Other formulas are built up from prime
formulas using logical constants $\bot$, $\wedge$, $\vee$, $\imp$,
$\forall x^{\rho}$, $\exists x^{\rho}$.
When a formula $A$ is derivable in $\HAw$, we write $\HAw \vdash A$.

\begin{notations*}
We use variables $k, l, m, n, \dots, x, y, z, \dots$ for objects of type $\Base$
and $\alpha, \beta, \gamma, \dots$ for objects of type $\Base \to \Base$.
We assume fixed bijective coding of $\Base^{*}$
 in $\Base$,
and identify finite sequences with their codes.
We use variables $a, b, c,\dots$ for finite sequences.

The empty sequence is denoted by $\nil$, and a singleton sequence is
denoted by $\seq{x^{\Base}}$.
Concatenation of finite sequences $a$ and $b$ is denoted by
$a * b $, and concatenation of finite sequence $a$ and a sequence
$\alpha$ is denoted by $a * \alpha$.  For a finite sequence $a$, its
length is denoted by $|a|$; if $n < |a|$ then $a_{n}$ denotes the
$n$-th entry of $a$.  For any $\alpha$ and $n$, we write
$\overline{\alpha}n$ for the initial segment of $\alpha$ of length
$n$.
We write $\AppendZero{a}$ for $a * (\lambda n. 0)$.
 
If $P$ and $Q$ are predicates on a type $\rho$,
we use abbreviations
\begin{align*}
  P \subseteq Q &\equiv \forall x^{\rho} \left[
  P(x) \imp Q(x) \right], &
  (P \cap Q)(x) &\equiv P(x) \wedge Q(x).
\end{align*}
Type superscripts $t^{\rho}$ and subscripts $=_{\rho}$ are omitted
whenever they can be inferred from the context.
\end{notations*}

\section{Bar theorem for closed terms of $\HAw$}\label{sec:Nhb}
We recall some technical notions that
are needed for our main theorem; see Troelstra and van Dalen
\cite[Chapter 4]{ConstMathI} for more details.
\begin{definition}\label{def:NbhFunc}
  A function $\gamma \colon
  \Base^{*} \to \Base$ is called a \emph{neighbourhood function} if
  \begin{enumerate}
        \item $\forall \alpha^{\Baire} \exists n^{\Base}
          \gamma(\overline{\alpha}n) > 0$,
        \item $\forall a^{\Base^{*}}
          \left[ \gamma(a) > 0 \imp \forall b^{\Base^{*}} \gamma(a) =
          \gamma(a*b)\right]$.
  \end{enumerate}
  Given a neighbourhood function $\gamma \colon \Base^{*} \to \Base$ and a
  function $f \colon \Baire  \to \Base$, we say that $\gamma$
  \emph{induces} $f$ if
  \[
    \forall \alpha^{\Baire} \forall n^{\Base} \left
    [ \gamma(\overline{\alpha}n) > 0 \imp
    f(\alpha) = \gamma(\overline{\alpha}n) \dotminus 1\right],
  \]
  where $\dotminus$ is the primitive recursive cut-off minus operation.
  We say that a function $f \colon \Baire \to \Base$
  \emph{has a neighbourhood function} if there exists a neighbourhood
  function $\gamma \colon \Base^{*} \to \Base$ which induces $f$.
  In this case, we also say that $\gamma$ is a neighbourhood function of
  $f$. 
\end{definition}
Note that a function $f \colon \Baire \to \Base$ may have many
different neighbourhood functions.  Moreover, if $f$ has a
neighbourhood function, then $f$ is continuous. 

\begin{definition}
  A predicate $P$ on  $\Base^{*}$ is 
  \begin{itemize}
    \item a \emph{bar} if $\forall \alpha^{\Baire}
      \exists n^{\Base} P(\overline{\alpha}n)$;
    \item \emph{decidable} if $\forall a^{\Base^{*}}
      \left[  P(a) \vee \neg P(a)\right]$;
    \item \emph{monotone} if 
      $\forall a^{\Base^{*}} \forall b^{\Base^{*}} \left[ P(a) \imp
      P(a*b)\right]$.
  \end{itemize}
  A neighbourhood function $\gamma \colon \Base^{*} \to \Base$
  determines a decidable monotone bar $S_{\gamma}$ by
  \begin{equation}\label{eq:USecSeq}
    S_{\gamma}(a) \equiv \gamma(a) > 0.
  \end{equation}
  We say that $\gamma$
  satisfies \emph{bar induction} if for any predicate $Q$ on
  $\Base^{*}$
  \begin{equation}\label{eq:BI}
    S_{\gamma} \subseteq Q \wedge \Ind(Q)\imp Q(\nil),
  \end{equation}
  where 
  \begin{equation*}
    \Ind(Q) \equiv \forall a^{\Base^{*}} \left[ \forall n^{\Base}
    Q(a * \langle n \rangle) \imp Q(a) \right].
  \end{equation*}
  A predicate $Q$ on $\Base^{*}$ for which $\Ind(Q)$ holds is said to
  be \emph{inductive}.
\end{definition}

\begin{lemma} \label{lem:inductive}
  Let $Q$ be a predicate on $\Base^{*}$. Then
  \[
    \forall k^{\Base} \forall a^{\Base^{*}}
    \left[ \Ind(Q) \wedge \forall b^{\Base^{*}}
      \left( \lh{b}=k \imp Q(a*b)  \right)
      \imp Q(a)\right].
  \]
\end{lemma}
\begin{proof}
  By a straightforward induction on $k$.
\end{proof}

We can now introduce our main theorem.
\begin{theorem}[Bar theorem]\label{thm:BarCont}
  For any closed term $Y \colon \Baire \to \Base$ of
  $\HAw$, there is a closed term $\gamma \colon \Base^{*} \to \Base$
  of $\HAw$
  such that 
  \begin{enumerate}
  \item  $\HAw \vdash \text{$\gamma$ is a neighbourhood function of
    $Y$}$,
  \item for any predicate $Q$ on ${\Base^{*}}$
    \begin{equation*}
      \HAw \vdash S_{\gamma} \subseteq Q \wedge \Ind(Q)\imp Q(\nil).
    \end{equation*}
  \end{enumerate}
\end{theorem}

The proof of Theorem \ref{thm:BarCont} is given in Section
\ref{sec:ProofBarThm}.

\section{Non-standard representation of terms of $\HAw$}
\label{sec:NonStdRepresentation}
Escard\'o \cite{EscardoEffectfulForcing} showed that every definable
function $Y \colon \Baire \to \Base$ of G\"odel's system~$\SystemT$
can be represented by a dialogue tree.
Escard\'o presented his result as a property of $\SystemT$-definable
function in the set-theoretical model of system~$\SystemT$. 
Here, we formalise his result in $\HAw$, but we abstract
away from the concrete model of dialogue trees.

\begin{proposition}\label{prop:MonadicFramework}
Suppose that we have a type $T\Base$ and closed terms
\begin{align*}
  \eta &: \Base \to T\Base,\\
  \KE_{\Base} &: (\Base \to T\Base) \to (T\Base \to T\Base),\\
  \At &: \Base \to T\Base,\\
  \_ \bullet \_ &: T\Base \to \Baire \to \Base,
\end{align*}
which satisfy the following equations:
\begin{align}
  \label{Axiom:eta}
  \eta(n) \bullet \alpha &= n,\\
  \label{Axiom:KE}
  f(\gamma \bullet \alpha) \bullet \alpha
  &= \KE_{\Base} (f)(\gamma) \bullet \alpha,\\
  \label{Axiom:At}
  \At(n) \bullet \alpha &= \alpha (n).
\end{align}
Then, for each closed term $Y \colon \Baire \to \Base$ of
$\HAw$, there exists a closed term $\gamma \colon T\Base$ such that
$
\forall \alpha^{\Baire} \gamma \mathop{\bullet} \alpha =
Y\alpha.
$
\end{proposition}
The term $\_\bullet\_$ allows us to regard an object of type
$T\Base$ as a function from $\Baire$ to $\Base$. Thus, the proposition says that every
closed
$\HAw$-term of type $\Baire \to \Base$ can be represented by an object
of type $T\Base$. Intuitively, a term of type $T\Base$ derived from a
closed term $Y \colon \Baire \to \Base$ tells us more
about the computation of~$Y$.

\subsection{Proof of Proposition \ref{prop:MonadicFramework}}
Let $\mathcal{T}_{\Omega}$ be the set of terms in contexts in an
indeterminate $\Omega \colon {\Base \to \Base}$, i.e.,
$\mathcal{T}_{\Omega}$ is defined by the rule described in
\eqref{def:TermInCont} but with an extra constant $\Omega \colon {\Base \to
\Base}$. A term in indeterminate $\Omega$ will be written as
$t[\Omega]$. We define two interpretations of $\mathcal{T}_{\Omega}$
in $\HAw$: one is a standard interpretation; the
other is a non-standard interpretation into the type structure over $T\Base$.

\subsubsection{Standard interpretation}
In the standard interpretation, a term in context 
$\Gamma \vdash t[\Omega]\colon{\rho} \in \mathcal{T}_{\Omega}$ is interpreted as $\Gamma
\vdash \lambda \alpha. t[\alpha/\Omega] \colon \Baire \to \rho$, where
$t[\alpha/\Omega]$ is a substitution of $\alpha$ for $\Omega$ in $t$. 

\begin{remark}\label{rem:Monad}
The interpretation is \emph{standard} in the following sense: from
the viewpoint of categorical logic \cite{PittsCategoricalLogic}, the
interpretation determines an equivalence between the two categories:
the category $\Con[\Omega]$ of contexts and terms in indeterminate
$\Omega$; and the Kleisli category of a monad $T_{\Baire}$ on the
category $\Con$ of contexts and terms in the original language, where
$T_{\Baire}$ is defined as follows:
\[
  \begin{aligned}
    T_{\Baire}(x_{0} \colon {\rho_{0}}, \dots, x_{n-1} \colon {\rho_{n-1}}) 
    &:=
    x_{0}' \colon {\rho_{0}^{\Baire}}, \dots, x_{n-1}' \colon
    {\rho_{n-1}^{\Baire}} \\
    T_{\Baire}(\Gamma \vdash t\colon{\rho}) 
    &:=
    T_{\Baire}(\Gamma) \vdash \lambda \alpha. t[x_{0}'(\alpha)/x_{0},\dots,
    x_{n-1}'(\alpha)/x_{n-1}] \colon \rho^{\Baire}
  \end{aligned}
\]
See Lambek \cite[Section 5]{LambekFunctionalCompletenss}.  Thus, the
standard interpretation transforms a term in context in indeterminate
$\Omega$ into an essentially equivalent representation expressed in
the original language.
\end{remark}

\subsubsection{Non-standard interpretation}
\label{sec:NonStdInterpretation}
Let $(T\Base, \eta, \KE_{\Base}, \At, \_\bullet\_)$ be the structure specified 
in Proposition \ref{prop:MonadicFramework}.
We define a translation $\rho \mapsto \lift{\rho}$ of
the standard type structure over $\Base$ into the non-standard type
structure over $T\Base$ as follows:
\begin{align*}
  \lift{\Base} &:=  T\Base
  \\
  \lift{(\rho \to \sigma)} &:= \lift{\rho} \to \lift{\sigma}
\end{align*}
Each context $\Gamma \equiv x_{0} \colon
\rho_{0},\dots, x_{n-1} \colon \rho_{n-1}$ is translated to
a context
$\lift{\Gamma} \equiv \lift{x_{0}} \colon
\lift{\rho_{0}},\dots, \lift{x_{n-1}} \colon \lift{\rho_{n-1}}$,
where we assume a fixed assignment $\lift{x} \colon \lift{\rho}$ of
a variable to each variable $x \colon \rho$.
Then, a term in context $\Gamma \vdash t[\Omega] \colon \rho$ 
is translated to a term in context $\lift{\Gamma} \vdash
\lift{t[\Omega]} \colon \lift{\rho}$ of $\HAw$ as follows:
\begin{align*}
  \lift{( \Gamma, x \colon \rho, \Delta \vdash x \colon \rho )}
  &:=  \lift{\Gamma}, \lift{x} \colon \lift{\rho},
  \lift{\Delta} \vdash \lift{x} \colon \lift{\rho} \\
  \lift{(\Gamma \vdash \mathtt{C}\colon\rho )}
  &:=  \lift{\Gamma} \vdash \lift{\mathtt{C}} \colon \lift{\rho} \\
  \lift{(\Gamma \vdash \lambda x^{\rho}. t^{\sigma} \colon \rho \to
\sigma)}
&:=  \lift{\Gamma} \vdash \lambda \lift{x}. \lift{t} \colon \lift{\rho}
\to \lift{\sigma} \\
\lift{(\Gamma \vdash u^{\rho \to \sigma}v^{\rho}\colon \sigma)}
&:= \lift{\Gamma} \vdash \lift{u} \lift{v} \colon \lift{\sigma}
\end{align*}
Here, each constant $\mathtt{C}$ is translated as follows:
\begin{align*}
  \lift{0} &:= \eta(0) \\
  \lift{\Succ} &:= \KE_{\Base}(\lambda x^{\Base}. \eta(\Succ(x)))\\
  \lift{\Omega} &:= \gen := \KE_{\Base} (\lambda x^{\Base}. \At(x))\\
  \lift{\Rec_{\rho}} &:= \lambda u^{\rho^{\dagger}}. \lambda
  F^{\lift{\Base} \to \lift{\rho} \to \lift{\rho}}.
  \KE_{\rho}(\Rec_{\lift{\rho}}(u, \lambda x^{\Base}. F(\eta(x))))
\end{align*}
where for higher types, we define
\[
  \KE_{\rho \to \sigma} := \lambda f^{\Base \to \lift{\rho} \to
  \lift{\sigma}}. \lambda u^{\lift{\Base}}. \lambda v^{\lift{\rho}}.
  \KE_{\sigma}(\lambda x^{\Base}. fxv)u.
\]

\subsubsection{Relating two interpretations}
We relate two interpretations by a logical relation.  Define a binary
predicate $\sim_{\rho}$ on $\rho^{\dagger}$ and $\Baire \to \rho$ by
induction on types:
\begin{equation}
  \label{def:LogicalRel}
  \begin{aligned}
    \gamma \sim_{\Base} f &\equiv
    \forall \alpha^{\Baire} \left[\gamma \mathop{\bullet} \alpha =
    f\alpha  \right],\\
    G \sim_{\rho \to \sigma} F
    &\equiv \forall x^{\rho^{\dagger}}\forall y^{\Baire \to
    \rho}  
    \left[
      x \sim_{\rho}  y \imp Gx \sim_{\sigma}  \lambda \alpha. F
    \alpha (y \alpha)  \right].
  \end{aligned}
\end{equation}
At the base type, $\gamma \sim_{\Base} f $ means that $\gamma$
represents $f$. At higher types, the definition of $\sim_{\rho \to
\sigma}$ requires application to respect the relation $\sim$:
on the left, $Gx$ is the application in 
the type structure over $T\Base$; on the right,
$\lambda \alpha. F \alpha (y \alpha)$ corresponds to the application 
in the Kleisli category of the monad $T_{\Baire}$ (see Remark
\ref{rem:Monad}).%
\footnote{If we regard $y^{\Baire \to \rho}$ and
$F^{\Baire \to (\rho \to \sigma)}$ as global points
in the Kleisli category of $T_{\Baire}$, then $\lambda \alpha. F \alpha (y \alpha)$
is obtained by composing $\seq{F,y}$ with the term 
$u\colon \sigma^{\rho}, v \colon \rho \vdash\lambda \alpha.
uv \colon \Baire \to \sigma$, which represents an evaluation morphism.}

\begin{lemma}\label{prop:LogicalRelation}
  For any term $\Gamma \vdash t[\Omega]:{\rho}$
  in indeterminate $\Omega$ and context
  $\Gamma \equiv x_{0}^{\rho_{0}}, \dots, x_{n-1}^{\rho_{n-1}}$,
  \begin{gather}
    \notag
  \begin{multlined}
    \HAw \vdash
    \forall u_{0}^{\rho_{0}^{\dagger}} \cdots \forall u_{n-1}^{\rho_{n-1}^{\dagger}}
    \forall y_{0}^{\Baire \to \rho_{0}} \cdots
    \forall y_{n-1}^{\Baire \to\rho_{n-1}}
    \Bigl[
    u_{0} \sim_{\rho_{0}} y_{0} \wedge \cdots 
    \wedge u_{n-1} \sim_{\rho_{n-1}} y_{n-1} \\
    \imp
     t^{\dagger}[\overline{u}/\lift{\overline{x}}] 
    \sim_{\rho}
    \lambda \alpha^{\Baire}.  t[\alpha/\Omega,\overline{y}(\alpha)/\overline{x}]
  \Bigr]
  \end{multlined}\\
  \shortintertext{where}
  \label{eq:substitution}
  \begin{aligned}
    t^{\dagger}[\overline{u}/\lift{\overline{x}}] 
    &\equiv
    t^{\dagger}[u_{0}/\lift{x}_{0},\dots,u_{n-1}/\lift{x}_{n-1}],\\
    t[\alpha/\Omega,\overline{y}(\alpha)/\overline{x}]
    &\equiv
    t[\alpha/\Omega,y_{0}(\alpha)/x_{0},\ldots,y_{n-1}(\alpha)/x_{n-1}].
  \end{aligned}
  \end{gather}
\end{lemma}
The proof of Lemma \ref{prop:LogicalRelation} relies on the
following lifting property.

\begin{lemma}
  \label{lem:TranslationKE}
  For each type $\rho$,
  \begin{multline*}
   \HAw \vdash \forall g^{\Base \to \rho^{\dagger}} \forall f^{\Base \to
    \Baire \to \rho}
    \left[ \forall n^{\Base} g(n) \sim_{\rho} f(n) 
      \imp
       \KE_{\rho}(g)\sim_{\Base \to \rho} \lambda \alpha.
       \lambda n. fn\alpha \right].
  \end{multline*}
\end{lemma}
\begin{proof}
  By induction on types.
  \medskip

  \noindent$\rho \equiv \Base$:
  Fix $g^{\Base \to \lift{\Base}}$ and $f^{\Base \to \Baire \to \Base}$, and suppose 
  that
  $\forall n^{\Base} g(n) \sim_{\Base} f(n)$.
  Let $\gamma^{\Base^{\dagger}}$ and $h^{\Baire \to \Base}$
  satisfy $\gamma \sim_{\Base} h$. Then, for any $\alpha$,
  \begin{align*}
    \KE_{\Base}(g)(\gamma) \bullet \alpha
    &=  g(\gamma \bullet \alpha) \bullet \alpha && \text{by
     \eqref{Axiom:KE}}\\
    &= f(\gamma \bullet \alpha)(\alpha)
    && \text{by $g(\gamma \bullet \alpha) \sim_{\Base} f(\gamma \bullet \alpha)$}
    \\
    &= f(h\alpha)\alpha
    && \text{by $\gamma \sim_{\Base} h$}
    \\
    &= (\lambda n. fn\alpha) (h\alpha).
  \end{align*}
  Thus $\KE_{\Base}(g)(\gamma) \sim_{\Base} \lambda \alpha. (\lambda
  n. fn\alpha) (h\alpha)$. Hence $\KE_{\Base}(g) \sim_{\Base \to
  \Base} \lambda \alpha. (\lambda n. fn\alpha)$.
  \medskip

  \noindent$\rho \equiv \sigma \to \tau$:
  Fix $g^{\Base \to (\sigma \to \tau)^{\dagger}}$ and $f^{\Base \to
  \Baire \to \sigma \to \tau}$, and suppose 
  that
  $\forall n^{\Base} g(n) \sim_{\sigma \to \tau} f(n)$.
  Let $\gamma^{\Base^{\dagger}}$ and $h^{\Baire \to \Base}$
  satisfy $\gamma \sim_{\Base} h$. We must show that
  \[
    \KE_{\sigma \to \tau}(g)(\gamma) \sim_{\sigma \to \tau} \lambda
    \alpha. f(h(\alpha))\alpha,
  \]
  where  $\KE_{\sigma \to \tau}(g)(\gamma) = \lambda u^{\sigma^{\dagger}}.
  \KE_{\tau}(\lambda x^{\Base}. gxu)\gamma$. To this end, fix
  $u^{\sigma^{\dagger}}$ and $y^{\Baire \to \sigma}$, and
  suppose that $u \sim_{\sigma} y$. We must show that
  \[
    \KE_{\tau}(\lambda x. gxu)\gamma \sim_{\tau} \lambda \alpha.  f(h(\alpha))\alpha(y\alpha).
  \]
  By induction hypothesis, it suffices to show that
  \[
    \forall n^{\Base} gnu \sim_{\tau} \lambda \alpha.
    fn\alpha(y\alpha).
  \]
  But this follows from the assumptions $\forall n^{\Base} g(n) \sim_{\sigma
  \to \tau} f(n)$ and $u \sim_{\sigma} y$.
\end{proof}

\begin{proof}[Proof of Lemma \ref{prop:LogicalRelation}.]
  By induction on terms in contexts.
  \medskip

  \noindent $\Gamma, x^{\rho}, \Delta \vdash x\colon{\rho}$: Trivial.
  \medskip

  \noindent $\Gamma \vdash \mathtt{C} \colon {\rho}$: We deal with each
  constant:
  \medskip

  \noindent $0^{\Base}$: We must
  show $\eta(0) \bullet \alpha = 0$, which follows from \eqref{Axiom:eta}.
  \medskip

  \noindent $\Succ^{\Base \to \Base}$: 
  By Lemma \ref{lem:TranslationKE}, it suffices to show that
  $\forall n^{\Base} \eta(\Succ(n)) \sim_{\Base} \lambda
  \alpha.\Succ(n)$.
  This follows from \eqref{Axiom:eta}.
  \medskip 

  \noindent $\Omega^{\Base \to \Base}$: 
  By Lemma \ref{lem:TranslationKE}, it suffices to show that
  $\forall n^{\Base} \At(n) \sim_{\Base} \lambda \alpha.
  \alpha(n)$. This follows from \eqref{Axiom:At}.
  \medskip

  \noindent $\Rec_{\rho}$: Let $u^{\rho^{\dagger}}$ and 
  $y^{\Baire \to \rho}$ such that $u \sim_{\rho} y$, and let
  $F^{\Base^{\dagger} \to \rho^{\dagger} \to \rho^{\dagger}}$ and
  $f^{\Baire \to \Base \to \rho \to \rho}$ such that
  $F \sim_{\Base \to \rho \to \rho} f$.  We must show that
  \[
    \KE_{\rho}(\Rec_{\rho^{\dagger}}(u,\lambda x.
    F(\eta(x)))) \sim_{\Base \to \rho} \lambda
    \alpha.\Rec(y\alpha, f\alpha).
  \]
  By Lemma \ref{lem:TranslationKE}, it suffices to show that
  \[
    \forall n^{\Base}\, \Rec_{\rho^{\dagger}}(u,\lambda x.
    F(\eta(x)))n \sim_{\rho} 
    \lambda \alpha.\Rec(y\alpha, f\alpha)n,
  \]
  which follows by a straightforward induction on $n$.
  \medskip

  \noindent $\Gamma \vdash \lambda x^{\rho}. t^{\sigma} \colon \rho \to
  \sigma$:  Immediate from induction hypothesis for $\Gamma, x^{\rho}
  \vdash t: \sigma$. 
  \medskip

  \noindent $\Gamma \vdash u^{\rho \to \sigma}v^{\rho}$:
  Immediate from induction hypothesis.
\end{proof}
We now complete the proof of the representation theorem.
\begin{proof}[Proof of Proposition \ref{prop:MonadicFramework}]
  Let $Y \colon \Baire \to \Base$ be a closed term of $\HAw$.
  Then $Y\Omega$ is a closed term of type $\Base$ in indeterminate
  $\Omega$. Put $\gamma := \lift{(Y\Omega)} = \lift{Y}\gen$. 
  By Lemma \ref{prop:LogicalRelation}, we have
  $\gamma \sim_{\Base} \lambda \alpha. Y\alpha$, i.e., 
  $\forall \alpha^{\Baire} \gamma \bullet \alpha = Y\alpha$.
\end{proof}

\subsection{Examples}
We give some examples of the structure $(T\Base, \eta, \KE_{\Base},
\At, \_\bullet\_)$ specified in Proposition \ref{prop:MonadicFramework}.
\subsubsection{Dialogue trees}\label{sec:Dialogue}
We show how Escard\'o's dialogue model 
\cite{EscardoEffectfulForcing} fits into the framework of Proposition~\ref{prop:MonadicFramework}.
The type $\DT$ of dialogue trees has two constructors
\begin{align*}
  \eta &\colon \Base \to \DT,\\
  \mathtt{D} &\colon \Base \to (\Base \to \DT) \to \DT.
\end{align*}
The constructor $\eta$ creates a leaf node of a tree labelled by a
natural number, which represents a possible result of the computation.
The constructor $\mathtt{D}$ creates an internal node which is labelled by a
natural number and has countably many branches. Internal nodes guide the
computation toward the leaves; see definition of $\_\bullet\_$ below.

The recursor $\mathcal{R}^{\DT}_{\rho}$ (for each type $\rho$) of dialogue trees has a type
\[
   (\Base \to \rho) \to
   \left( \Base \to \left( \Base \to \DT \right) \to \left( \Base \to
   \rho \right) \to \rho \right) \to \DT \to \rho,
\]
and satisfies the following defining equations:
\begin{align*}
  \mathcal{R}^{\DT}_{\rho}(u,f,\eta(n)) &= u(n),\\
  \mathcal{R}^{\DT}_{\rho}(u,f,\mathtt{D}n\varphi)
  &= fn\varphi(\lambda x^{\Base}.
  \mathcal{R}^{\DT}_{\rho}(u,f,\varphi(x))).
\end{align*}
With the help of recursors,
functions $\KE_{\Base} \colon (\Base \to \DT) \to (\DT \to \DT)$ and
$\_\bullet\_ \colon \DT \to \Baire \to \Base$ are defined as
\begin{align*}
  \KE_{\Base}(f,\eta(n)) &= f(n), \\
  \KE_{\Base}(f,\mathtt{D}n\varphi)&= \mathtt{D}n(\lambda x^{\Base}.
  \KE_{\Base}(f, \varphi(x))), \\
  \eta(n) \bullet \alpha &= n,\\
  \mathtt{D}n\varphi \bullet \alpha &= \varphi(\alpha(n)) \bullet \alpha.
\end{align*}
Finally, $\At \colon \Base \to \DT$ is defined as $\At := \lambda
x^{\Base}.  \mathtt{D}x\eta$.  

Let $\HAw + \DT$ be an extension of $\HAw$ with the type $\DT$ of
dialogue trees as an extra base type.
The extension includes the constructors and recursors of
dialogue trees, the defining equations of the recursors, and the
following induction schema for dialogue trees:
\begin{equation*}
  \forall x^{\Base} A(\eta(x)) \wedge \forall x^{\Base} \forall
  \varphi^{\Base \to \DT} \left[ \forall n^{\Base}  A(\varphi(n)) \imp A(\mathtt{D}x\varphi)\right]
  \imp \forall \gamma^{\DT} A(\gamma).
\end{equation*}
In $\HAw + \DT$, one can show that the structure
$(\DT, \eta, \KE_{\Base}, \At, \_\bullet\_)$ satisfies
\eqref{Axiom:eta}, \eqref{Axiom:KE}, and \eqref{Axiom:At} (cf.\  Escard\'o \cite[Section
3]{EscardoEffectfulForcing}). Hence, Proposition
\ref{prop:MonadicFramework} instantiated with the type of dialogue
trees is valid in $\HAw + \DT$: every closed term $Y \colon
\Baire \to \Base$ of $\HAw$ can be represented by a dialogue tree.

\subsubsection{Brouwer-operations}\label{sec:BrouwerTree}
Dialogue trees are convenient for studying continuity properties of
definable functions.  For the proof of bar theorem, however,
Brouwer-operations (inductively defined neighbourhood functions \cite[Chapter 4,
Section 8.4]{ConstMathI}) are more suitable.  

The type $\BT$ of Brouwer-operations has two constructors
\begin{align*}
  \eta &\colon \Base \to \BT,\\
  \sup &\colon (\Base \to \BT) \to \BT,
\end{align*}
where $\eta$ creates a leaf node labelled by a natural number, and 
$\sup$ creates an internal node from countably many subtrees.

The recursor $\mathcal{R}^{\BT}_{\rho}$ (for each type $\rho$) of Brouwer-operations has a type
\[
   (\Base \to \rho) \to
   \left( \left( \Base \to \BT \right) \to \left( \Base \to
   \rho \right) \to \rho \right) \to \BT \to \rho,
\]
and satisfies the defining equations:
\begin{align*}
  \mathcal{R}^{\BT}_{\rho}(u,f,\eta(n)) &= u(n),\\
  \mathcal{R}^{\BT}_{\rho}(u,f,\sup \varphi)
  &= f\varphi(\lambda x^{\Base}.  \mathcal{R}^{\BT}_{\rho}(u,f,\varphi(x))).
\end{align*}
With the help of recursors,
functions $\KE_{\Base} \colon (\Base \to \BT) \to (\BT \to \BT)$ and
$\_\bullet\_ \colon \BT \to \Baire \to \Base$ are defined as
\begin{align*}
  \KE_{\Base}(f,\gamma) &= \mathsf{Aux}(f,\gamma,\nil),\\
  \eta(n) \mathop{\bullet} \alpha &= n, \\
  \sup \varphi \mathop{\bullet} \alpha &= \varphi(\alpha(0)) \bullet
  (\lambda x. \alpha(\Succ(x))).
\end{align*}
Here $\mathsf{Aux} \colon (\Base \to \BT)
\to \BT \to \Base^{*} \to \BT$ is a function defined with a help
of another auxiliary function  $\mathsf{skip} \colon \BT \to
\Base^{*} \to \BT$ as follows:
\begin{align*}
  \mathsf{Aux}(f,\eta(n),a) &= \mathsf{skip}(f(n),a), \\
  \mathsf{Aux}(f,\sup \varphi, a) &= \sup (\lambda x^{\Base}. \mathsf{Aux}(f,
  \varphi(x), a * \seq{x})),\\
  \mathsf{skip}(\gamma, \nil) &= \gamma,\\
  \mathsf{skip}(\eta(n), \seq{x} * a) &= \mathsf{skip}(\eta(n), a),\\
  \mathsf{skip}(\sup \varphi, \seq{x} * a) &=
  \mathsf{skip}(\varphi(x), a).
\end{align*}
Finally, $\At \colon \Base \to \BT$ is defined by a primitive recursion:
\begin{align*}
  \At(0) &= \sup \eta, \\
  \At(\Succ(n)) &= \sup (\lambda x^{\Base}.\At(n)).
\end{align*}

Let $\HAw + \BT$ be an extension of $\HAw$ with the type $\BT$ of
Brouwer-operations as an extra base type.%
\footnote{$\HAw + \BT$ is similar to system~$\SystemT_{1}$
described in Zucker \cite[6.3.6 (b)]{ZuckerIteratedIDTreeOrdinal}.}
The extension includes the constructors and recursors of
Brouwer-operations, the defining equations of the recursors, and the
following induction schema for Brouwer-operations:
\begin{equation*}
  \forall x^{\Base} A(\eta(x)) \wedge \forall \varphi^{\Base \to \BT}
  \left[ \forall n^{\Base}  A(\varphi(n)) \imp A(\sup \varphi)\right] \imp \forall \gamma^{\BT} A(\gamma).
\end{equation*}
In $\HAw + \BT$, one can show that
\[
  f(\gamma \bullet \alpha) \bullet (a * \alpha)
  =
  \mathsf{Aux}(f,\gamma, a) \bullet \alpha 
\]
for all $f$, $\gamma$, $a$, and $\alpha$, from
which \eqref{Axiom:KE} follows. Conditions \eqref{Axiom:eta} and
\eqref{Axiom:At} are easy to check.  Hence, Proposition
\ref{prop:MonadicFramework} instantiated with the type of
Brouwer-operations is valid in $\HAw + \BT$.

Furthermore, Theorem \ref{thm:BarCont} holds
in $\HAw + \BT$: for any closed term $Y \colon \Baire \to \Base$ of $\HAw$,
there is a closed term $\gamma \colon \Base^{*} \to \Base$ of $\HAw +
\BT$ such that
\begin{enumerate}
\item  $\HAw + \BT \vdash \text{$\gamma$ is a neighbourhood function of
  $Y$}$,
\item 
    $
    \HAw + \BT \vdash S_{\gamma} \subseteq Q \wedge \Ind(Q)\imp
    Q(\nil)$ for any predicate $Q$ on ${\Base^{*}}$.
\end{enumerate}
To see this, first, each Brouwer-operation $\gamma$ determines a neighbourhood
function $\delta(\gamma) \colon \Base^{*} \to \Base$ as follows:
\begin{align*}
  \delta(\eta(n))(a) &= \Succ(n),\\
  \delta(\sup \varphi)(\nil) &= 0,\\
  \delta(\sup \varphi)(\seq{x} * a) &= \delta(\varphi(x))(a).
\end{align*}
It is easy to see that if 
a Brouwer-operation $\gamma$
represents a function $f \colon \Baire \to \Base$, i.e., $\forall
\alpha^{\Baire} \gamma \bullet \alpha = f(\alpha)$, then 
$\delta(\gamma)$ is a neighbourhood function of $f$.

Second, it is well known that $\delta(\gamma)$ satisfies induction
over unsecured sequences~\cite[Chapter 4, Proposition
8.12]{ConstMathI}, i.e., 
  $
  S_{\delta(\gamma)} \subseteq Q  \wedge \Ind(Q) \imp Q(\nil)
  $
for any predicate $Q$ on $\Base^{*}$.

\subsection{Non-standard representation as a model construction}
\label{sec:NonStdAsModels}
  If we think of Proposition \ref{prop:MonadicFramework} as a 
  set-theoretical model construction as in Escard\'o
  \cite{EscardoEffectfulForcing} --- reading $\Base$ as the set of
  natural numbers, $T\Base$ as another set,  and $\eta, \KE_{\Base},
  \At$, and $\_\bullet\_$ as functions --- then we obtain a family of
  representation theorems for system~$\SystemT$ definable functions
  from $\Baire$ to $\Base$. Specifically, any $\SystemT$-definable
  set-theoretical function $f \colon \Baire \to \Base$ can be
  represented by an element of $T\Base$.
  In the following, we assume this set-theoretical reading of
  Proposition~\ref{prop:MonadicFramework}. The argument in this subsection
  can be carried out in a suitable constructive set theory, e.g.,
  Aczel's $\CZF$ \cite{Aczel-Rathjen-Note} extended with generalised
  inductive definitions.
  
  If one is interested in strong continuity properties of
  $\SystemT$-definable functions, then the dialogue model of Escard\'o
  \cite{EscardoEffectfulForcing} or set-theoretical version of
  Brouwer-operation model presented in Section \ref{sec:BrouwerTree}
  seem to be most suitable. These models allow us to show, for
  example, uniform continuity of $\SystemT$-definable functions 
  from $\Baire$ to $\Base$ on the Cantor space (cf.\ Section
  \ref{sec:UCont}).  Note, however,
  that the constructions of these models require generalised inductive
  definitions.

  If one is merely interested in point-wise continuity of
  $\SystemT$-definable functions, then one can use the following structure:
  \begin{equation}
    \label{eq:ContModel}
    \begin{aligned}
      T\Base &:= \left\{ f \colon \Baire \to \Base \mid \text{$f$ is
        point-wise continuous} \right\} \\
      \eta &:= \lambda n. \lambda \alpha. n \\
      \KE_{\Base} &:= \lambda f.\lambda g.\lambda \alpha.
      f(g(\alpha))(\alpha) \\
      \At &:= \lambda n.\lambda \alpha. \alpha(n) \\
      \_\bullet\_ &:= \lambda f. \lambda \alpha. f(\alpha)
    \end{aligned}
  \end{equation}
  where the lambda notations in \eqref{eq:ContModel} should be read 
  set-theoretically.\footnote{This model was suggested by Mart{\'\i}n Escard\'o.}
  With this model, one has that every $\SystemT$-definable function
  from $\Baire$ to $\Base$ is point-wise continuous.

  If one is interested in stronger continuity properties in the
  absence of generalised inductive definitions, one may
  instantiate the structure $(T\Base, \eta,
  \KE_{\Base},\At,{\_\bullet\_})$ with neighbourhood functions instead
  of Brouwer-operations:
  \begin{equation}
    \begin{aligned}\label{eq:NbhModel}
      T\Base' &:= \left\{ \gamma \colon \Base^{*} \to \Base \mid
      \text{$\gamma$ is a neighbourhood function} \right\} \\
      \eta' &:= \lambda n. \lambda a. n + 1\\
      \KE_{\Base}' &:= \lambda f. \lambda
      \gamma.\lambda a.
      \sg(\gamma(a)) \cdot f(\gamma(a) \dotminus 1)(a)\\
      \At' &:= \lambda n. \lambda a.
      \begin{cases}
        0 & \text{if $\lh{a} \leq n$}\\
        a_{n} + 1 & \text{otherwise}
      \end{cases}\\
      \_ \bullet'\!\_ &:= \lambda \gamma. \lambda \alpha. 
      \gamma(\overline{\alpha}(\mu n. \gamma(\overline{\alpha}n) > 0)) \dotminus 1
    \end{aligned}
  \end{equation}
where $\cdot$ is the multiplication, $\sg$ is the signum
function, and $\mu n. A(n)$ is the bounded search function.
Note that in the definition of $\_ \bullet'\!\_$, the condition $\gamma(\overline{\alpha}n) > 0$
is satisfied for some $n$ because $\gamma$ is a neighbourhood
function.

The idea behind the ``neighbourhood function model'' presented in
\eqref{eq:NbhModel} is expressed in the following lemma.
\begin{lemma}\label{lem:NbhRepresenation}
  Let $(T\Base, \eta,
  \KE_{\Base},\At,{\_\bullet\_})$
  be the continuous model presented in $\eqref{eq:ContModel}$
  and let $(T\Base', \eta',
  \KE_{\Base}',\At',{\_\bullet'\!\_})$
  be the neighbourhood function model presented in
  \eqref{eq:NbhModel}. Then
  \begin{enumerate}
    \item For each $n \in \Base$, $\eta'(n)$ and $\At'(n)$
      are neighbourhood functions of $\eta(n)$ and $\At(n)$,
      respectively;

    \item If $h \colon \Base \to T\Base'$ and $f \colon \Base \to
      T\Base$ are functions such that $h(n)$ is a neighbourhood
      function of $f(n)$ for each $n \in \Base$, and if $\gamma$ is a
      neighbourhood function of $g$, then $\KE_{\Base}'(h)(\gamma)$ is a
      neighbourhood function of $\KE_{\Base}(f)(g)$;

    \item 
    If $\gamma$ is a neighbourhood function of $f$, then
      $
      \gamma \bullet' \alpha = f \bullet \alpha
      $
    for all $\alpha \in \Baire$.
  \end{enumerate}
\end{lemma}
\begin{proof}
  Straightforward.
\end{proof}
With the neighbourhood function model, one has that every
$\SystemT$-definable function from $\Baire$ to  $\Base$ has a
neighbourhood function.

\section{Proof of bar theorem}\label{sec:ProofBarThm}
The proof of bar theorem (Theorem \ref{thm:BarCont}) is based on
the set-theoretical neighbourhood function model presented in
\eqref{eq:NbhModel}.  However, since the type of ``neighbourhood
functions'' is not directly available in $\HAw$, we need to make some
adjustments to the proof of Proposition \ref{prop:MonadicFramework}.
Moreover, the domain of a neighbourhood function, unlike that of a
Brouwer-operation,  does not necessarily admit bar induction.
Nevertheless, when a neighbourhood function is presented as a concrete
term of $\HAw$, we can draw stronger properties from it using logical
relations and induction on terms. Those are the basic ideas of the
proof presented below.

First, we translate each type $\rho$ to the corresponding type in the
type structure over $\Base^{*} \to \Base$:
\begin{align*}
  \Base^{\dagger} &:= \Base^{*} \to \Base
  \\
  (\rho \to \sigma)^{\dagger} &:= \rho^{\dagger} \to \sigma^{\dagger}
\end{align*}
We think of $\Base^{\dagger}$ as the type of neighbourhood
functions. Obviously, $\Base^{\dagger}$ contains functions that are not
neighbourhood functions. We take care of this issue by
modifying the logical relation below.%

We translate each term in context $\Gamma \vdash t[\Omega] \colon \rho$ in
indeterminate $\Omega$ 
as in Section \ref{sec:NonStdInterpretation}, instantiating $\eta$,
$\KE_{\Base}$, and $\At$ with the following terms:
\begin{equation}
  \label{eq:FormalNbhModel}
\begin{aligned}
  \eta &:= \lambda x^{\Base}. \lambda a^{\Base^{*}}. \Succ(x)\\
  \KE_{\Base} &:= \lambda f^{\Base \to \lift{\Base}}. \lambda
  \gamma^{\lift{\Base}}.\lambda a^{\Base^{*}}.
  \sg(\gamma(a)) \cdot f(\gamma(a) \dotminus 1)(a)\\
  \At &:= \lambda n^{\Base}. \lambda a^{\Base^{*}}.
  \begin{cases}
    0 & \text{if $\lh{a} \leq n$}\\
    \Succ(a_{n}) & \text{otherwise}
  \end{cases}
\end{aligned}
\end{equation}
where $\cdot$ and $\sg$ are the primitive recursive multiplication and
signum function, respectively. Note that \eqref{eq:FormalNbhModel}
defines terms of $\HAw$ while \eqref{eq:NbhModel} defines
set-theoretical functions.

The main difference between the interpretation in this section 
and that of Section \ref{sec:NonStdInterpretation} is the lack of
term $\_\bullet\_ \colon \lift{\Base} \to \Baire \to \Base$.
Ideally, we would define it as a partial application
$\gamma \bullet \alpha := \gamma(\overline{\alpha}(\mu n.
\gamma(\overline{\alpha}n) > 0)) \dotminus 1$, which would be total
if $\lift{\Base}$ were the type of neighbourhood functions (which is
not).
 
We deal with the lack of function $\_\bullet\_$ by
modifying the logical relation \eqref{def:LogicalRel} as follows:
for each type $\rho$, define a binary predicate $\approx_{\rho}$ on
$\rho^{\dagger}$ and $\Baire \to \rho$ by induction on types:
\begin{align*}
  \gamma \approx_{\Base} f &\equiv \text{$\gamma$ is a neighbourhood
  function of \;$f$}, \\
  G \approx_{\rho \to \sigma} F
  &\equiv \forall x^{\rho^{\dagger}}\forall y^{\Baire \to \rho}  
  \left[
  x \approx_{\rho}  y \imp Gx \approx_{\sigma}  \lambda \alpha. F
  \alpha (y \alpha)  \right].
\end{align*}
The proof of Lemma \ref{lem:TranslationKE} has to be adapted to $\approx_{\rho}$.
\begin{lemma}
  \label{lem:RelTranslationKE}
  For each type $\rho$, 
  \begin{equation*}
    \HAw \vdash
    \forall g^{\Base \to \rho^{\dagger}} \forall f^{\Base \to \Baire \to \rho}
    \left[ \forall n^{\Base} g(n) \approx_{\rho} f(n) 
      \imp
      \KE_{\rho}(g)\approx_{\Base \to \rho} \lambda \alpha.
    \lambda n. fn\alpha \right].
  \end{equation*}
\end{lemma}
\begin{proof}
  By induction on types.
  \medskip

  \noindent$\rho \equiv \Base$:
  Fix $g^{\Base \to \Base^{\dagger}}$ and $f^{\Base \to
  \Baire \to \Base}$, and suppose 
  that
  $\forall n^{\Base} g(n) \approx_{\Base} f(n)$.
  Let $\gamma^{\Base^{\dagger}}$ and $h^{\Baire \to \Base}$
  satisfy $\gamma \approx_{\Base} h$. Put
  \begin{equation}\label{eq:delta}
    \delta := \KE_{\Base}(g)(\gamma) = \lambda a. \sg(\gamma(a)) \cdot
    g(\gamma(a)\dotminus 1)(a).
  \end{equation}
  We must show that $\delta$ is a neighbourhood function of $\lambda
  \alpha. f(h(\alpha))\alpha$.

  To this end, fix $\alpha$.
  Since $\gamma$ is a neighbourhood function, there exists
  an $n$ such that $\gamma(\overline{\alpha}n) > 0$.
  Put $i := \gamma(\overline{\alpha}n) \dotminus 1$. Since $g(i)$ is a
  neighbourhood function (of $f(i)$), there exists
  an $m$ such that $g(i)(\overline{\alpha}m) > 0$.
  Put $N := \max\{n,m\}$ and $a := \overline{\alpha}N$. Then
  $\gamma(a) > 0$, $i = \gamma(a) \dotminus 1$, and $g(i)(a) > 0$.
  Thus $\delta(a) > 0$.

  Next, suppose that $\delta(a) > 0$. Then $\gamma(a)
  > 0$ and $g(\gamma(a) \dotminus 1)(a) > 0$.
  Thus, for any $b$, 
  \[
    \begin{aligned}
      \delta(a*b)
      &= \sg(\gamma(a*b)) \cdot g(\gamma(a*b) \dotminus 1)(a*b) \\
      &= \sg(\gamma(a)) \cdot g(\gamma(a) \dotminus 1)(a*b)\\
      &= \sg(\gamma(a)) \cdot g(\gamma(a) \dotminus 1)(a)\\
      &= \delta(a).
    \end{aligned}
  \]
  Hence, $\delta$ is a neighbourhood function.

  Lastly, let $\alpha$ and $n$ 
  such that $\delta(\overline{\alpha}n)>0$.
  Then $\gamma(\overline{\alpha}n) > 0$, and so $h(\alpha) =
  \gamma(\overline{\alpha}n) \dotminus 1$.
  Since $g(\gamma(\overline{\alpha}n \dotminus 1))(\overline{\alpha}n)
  = g(h(\alpha))(\overline{\alpha}n) > 0$, we have 
  $f(h(\alpha))\alpha = g(h(\alpha))(\overline{\alpha}n) \dotminus 1 =
  \delta(\overline{\alpha}n) \dotminus 1$. Hence, $\delta$ induces
   $\lambda \alpha. f(h(\alpha))\alpha$.
  \medskip

  \noindent$\rho \equiv \sigma \to \tau$:
  The proof is exactly the same as in the inductive case of Lemma~\ref{lem:TranslationKE}.
\end{proof}

\begin{lemma}\label{lem:RelTranslation}
  For any term $\Gamma \vdash t[\Omega]\colon{\rho}$
  in indeterminate $\Omega$ and context
  $\Gamma \equiv x_{0}^{\rho_{0}}, \dots, x_{n-1}^{\rho_{n-1}}$,
  \begin{equation*}
    \begin{multlined}
      \HAw \vdash
      \forall u_{0}^{\rho_{0}^{\dagger}} \cdots \forall u_{n-1}^{\rho_{n-1}^{\dagger}}
      \forall y_{0}^{\Baire \to \rho_{0}}\!\! \cdots
      \forall y_{n-1}^{\Baire \to\rho_{n-1}}
      \Bigl[
        u_{0} \approx_{\rho_{0}} y_{0} \wedge \cdots 
        \wedge u_{n-1} \approx_{\rho_{n-1}} y_{n-1} \\
        \imp
        t^{\dagger}[\overline{u}/\lift{\overline{x}}] 
        \approx_{\rho}
        \lambda \alpha^{\Baire}.  t[\alpha/\Omega,\overline{y}(\alpha)/\overline{x}]
      \Bigr],
    \end{multlined}
  \end{equation*}
  where 
  $t^{\dagger}[\overline{u}/\lift{\overline{x}}]$ and 
  $t[\alpha/\Omega,\overline{y}(\alpha)/\overline{x}]$
  are defined as in \eqref{eq:substitution}.
\end{lemma}
\begin{proof}
  By induction on terms in contexts.
  \medskip

  \noindent $\Gamma, x^{\rho}, \Delta \vdash x\colon{\rho}$: Trivial.
  \medskip

  \noindent $\Gamma \vdash \mathtt{C} \colon {\rho}$: We deal with each
  constant:
  \medskip

  \noindent $0^{\Base}$: We must
  show that $0^{\dagger} = \eta(0) = \lambda a. \Succ(0)$ is a neighbourhood function of
  $\lambda \alpha. 0$, which is obvious.
  \medskip

  \noindent $\Succ^{\Base \to \Base}$: 
  By Lemma \ref{lem:RelTranslationKE}, it suffices to show that for
  each $n$, $\eta(\Succ(n)) = \lambda a. \Succ(\Succ(n))$ is a
  neighbourhood function of $\lambda \alpha.\Succ(n)$, which is
  obvious.
  \medskip 

  \noindent $\Omega^{\Base \to \Base}$: 
  By Lemma \ref{lem:RelTranslationKE}, it suffices to show that
  for each $n$, $\At(n)$ is a neighbourhood function of $\lambda \alpha.
  \alpha(n)$, which follows from the definition of $\At$.
  \smallskip

  \noindent $\Rec_{\rho}$: Let $u^{\rho^{\dagger}}$ and 
  $y^{\Baire \to \rho}$ such that $u \sim_{\rho} y$, and let
  $F^{\Base^{\dagger} \to \rho^{\dagger} \to \rho^{\dagger}}$ and
  $f^{\Baire \to \Base \to \rho \to \rho}$ such that
  $F \sim_{\Base \to \rho \to \rho} f$.  We must show that
  \[
    \KE_{\rho}(\Rec_{\rho^{\dagger}}(u,\lambda x.
    F(\eta(x)))) \sim_{\Base \to \rho} \lambda
    \alpha.\Rec(y\alpha, f\alpha).
  \]
  By Lemma \ref{lem:RelTranslationKE}, it suffices to show that
  \[
    \forall n^{\Base}\, \Rec_{\rho^{\dagger}}(u,\lambda x.
    F(\eta(x)))n \sim_{\rho} 
    \lambda \alpha.\Rec(y\alpha, f\alpha)n,
  \]
  which follows by a straightforward induction on $n$.
  \medskip

  \noindent $\Gamma \vdash \lambda x^{\rho}. t^{\sigma} \colon \rho \to
  \sigma$:  Immediate from induction hypothesis for $\Gamma, x^{\rho}
  \vdash t: \sigma$. 
  \medskip

  \noindent $\Gamma \vdash u^{\rho \to \sigma}v^{\rho}$:
  Immediate from induction hypothesis.
\end{proof}


For each predicate $Q$ on $\Base^{*}$, define
a predicate $P^{Q}_{\rho}$ on $\rho^{\dagger}$ by induction on types:
\begin{align*}
  P^{Q}_{\Base}(\gamma) &\equiv 
  \text{$\gamma$ is a neighbourhood function} \; \wedge\\
  &\qquad\qquad \forall a^{\Base^{*}} \left[ \forall b^{\Base^{*}} \left[S_{\gamma}(a*b) \imp
  Q(a*b)  \right] \wedge \Ind(Q) \imp Q(a)\right], \\
  P^{Q}_{\rho \to \sigma}(G) 
  &\equiv \forall x^{\rho^{\dagger}} \left[ P^{Q}_{\rho}(x)  \imp
  P^{Q}_{\sigma}(Gx)\right].
\end{align*}

First, we prove the following lifting property.
\begin{lemma}
  \label{lem:BarPropKE}
  For each type $\rho$, 
  \[
   \HAw \vdash  \forall g^{\Base \to \rho^{\dagger}}
    \left[ \forall n^{\Base} P^{Q}_{\rho}(g(n))
      \imp
    P^{Q}_{\Base \to \rho}(\KE_{\rho}(g)) \right].
  \]
\end{lemma}
\begin{proof}
  By induction on types.
  \medskip

  \noindent $\rho \equiv \Base$: Fix $g^{\Base \to \lift{\Base}}$ and 
  suppose that $\forall n^{\Base}\, P^{Q}_{\Base}(g(n))$. We must
  show that $P^{Q}_{\Base \to \Base}(\KE_{\Base}(g))$, i.e.,
  \[
    \forall \gamma^{\Base^{\dagger}}\left[ P^{Q}_{\Base}(\gamma) \imp
    P^{Q}_{\Base}(\KE_{\Base}(g)\gamma)\right].
  \]
  Let $\gamma^{\Base^{\dagger}}$ such that $P^{Q}_{\Base}(\gamma)$.
  Put 
  \[
    \delta := \KE_{\Base}(g)(\gamma) = \lambda a. \sg(\gamma(a)) \cdot
    g(\gamma(a) \dotminus 1)(a).
  \]
  As shown in the proof of Lemma \ref{lem:RelTranslationKE},
  $\delta$ is a neighbourhood function.
  Fix $a^{\Base^{*}}$, and suppose that 
  $
  \forall b^{\Base^{*}}\left[ S_{\delta}(a * b ) \imp Q(a*b)\right]
  $
  and that $Q$ is inductive.  We must show that $Q(a)$. Since
  $P^{Q}_{\Base}(\gamma)$, it suffices to show that 
  \[
    \forall b^{\Base^{*}}\left[ S_{\gamma}(a * b ) \imp
    Q(a*b)\right].
  \]
  Let $b^{\Base^{*}}$ such that
  $S_{\gamma}(a*b)$, and put $i := \gamma(a*b) \dotminus 1$. Since
  $P^{Q}_{\Base}(g(i))$, to see that
  $Q(a*b)$ holds, it suffices to show that $\forall c^{\Base^{*}}
  \left[  S_{g(i)}(a*b*c) \imp Q(a*b*c) \right]$. Let
  $c^{\Base^{*}}$ such that $S_{g(i)}(a*b*c)$.
  Then $S_{\delta}(a*b*c)$ and so $Q(a*b*c)$ as required. 
  \medskip

  \noindent $\rho \equiv \sigma \to \tau$:
  Fix $g^{\Base \to \lift{(\sigma \to \tau)}}$ and suppose that $\forall n^{\Base}P^{Q}_{\sigma \to \tau}(g(n))$.
  Let $\gamma^{\Base^{\dagger}}$ and $u^{\sigma^{\dagger}}$ such that
  $P^{Q}_{\Base}(\gamma)$ and 
  $P^{Q}_{\sigma}(u)$. We must show that
  $P^{Q}_{\tau}(\KE_{\tau}(\lambda n^{\Base}. g(n)u)\gamma)$.
  Since $\forall n^{\Base}\; P^{Q}_{\tau}(g(n)u)$, 
  we have 
  $P^{Q}_{\Base \to \tau}(\KE_{\tau}(\lambda n^{\Base}. g(n)u))$
  by induction hypothesis.  Therefore
  $P^{Q}_{\tau}(\KE_{\tau}(\lambda n^{\Base}. g(n)u)\gamma)$.
\end{proof}

\begin{lemma}\label{lem:BarProp}
  For any term $\Gamma \vdash t[\Omega]\colon{\rho}$
  in indeterminate $\Omega$ and context
  $\Gamma \equiv x_{0}^{\rho_{0}}, \dots,  x_{n-1}^{\rho_{n-1}}$,
  \begin{gather*}
    \HAw \vdash
    \forall u_{0}^{\rho_{0}^{\dagger}} \cdots \forall u_{n-1}^{\rho_{n-1}^{\dagger}}
    \left[
      P^{Q}_{\rho_{0}}(u_{0}) \wedge \cdots \wedge
      P^{Q}_{\rho_{n-1}}(u_{n-1})
      \imp
      P^{Q}_{\rho}(t^{\dagger}[\overline{u}/\lift{\overline{x}}])
    \right].
  \end{gather*}
\end{lemma}
\begin{proof}
  By induction on terms in contexts.
  \medskip

  \noindent $\Gamma, x^{\rho}, \Delta \vdash x\colon{\rho}$: Trivial.
  \medskip

  \noindent$\Gamma \vdash \mathtt{C} \colon {\rho}$: We deal with each constant
  below:
  \medskip 

  \noindent $0^{\Base}$: The term $0^{\dagger} \equiv \lambda a. \Succ(0)$ 
  obviously 
  satisfies the required property.
  \medskip

  \noindent $\Succ^{\Base \to \Base}$: 
  By Lemma \ref{lem:BarPropKE}, it suffices to show that
  $\forall n^{\Base}. P^{Q}_{\Base}(\lambda a. \eta(\Succ(n)))$, which is
  obvious.
  \medskip

  \noindent $\Omega^{\Base \to \Base}$: 
  By Lemma \ref{lem:BarPropKE}, it suffices to show that
  $\forall n^{\Base} P^{Q}_{\Base}(\At(n))$.
  Fix $n^{\Base}$. Then $\At(n)$ is a neighbourhood function.
  Let $a^{\Base^{*}}$ satisfy
  \[
    \forall b^{\Base^{*}}\left[ \At(n)(a*b) > 0 \to Q(a*b)\right]
  \]
  and suppose that $Q$ is inductive. Put $k := n \dotminus \lh{a} +
  1$. For each
  $b^{\Base^{*}}$ such that $\lh{b} = k$, we have $\At(n)(a*b) > 0$, and  so
  $Q(a*b)$. Thus, $Q(a)$ by Lemma \ref{lem:inductive}.
  \medskip

  \noindent $\Rec_{\rho}$: Let $u^{\rho^{\dagger}}$ and
  $F^{\Base^{\dagger} \to \rho^{\dagger} \to \rho^{\dagger}}$
  satisfy
  $P^{Q}_{\rho}$ and $P^{Q}_{\Base \to \rho \to \rho}$, respectively.
  We must show that
    $
    P^{Q}_{\Base \to \rho}(\KE_{\rho}(\Rec_{\rho^{\dagger}}(u,\lambda
    x.  F(\eta(x))))).
    $
  By Lemma \ref{lem:BarPropKE}, it suffices to show that
  $
  \forall n^{\Base} \;
  P^{Q}_{\rho}(\Rec_{\rho^{\dagger}}(u,\lambda x. F(\eta(x)))n).
  $
  This follows by a straightforward induction on $n$ using
  assumptions $P^{Q}_{\rho}(u)$ and $P^{Q}_{\Base \to \rho \to
  \rho}(F)$  and the fact that $P^{Q}_{\Base}(\eta(n))$
  for each $n^{\Base}$.\\[.5em]
  \noindent $\Gamma \vdash \lambda x^{\rho}. t^{\sigma} \colon \rho \to
  \sigma$:  Immediate from induction hypothesis for $\Gamma, x^{\rho}
  \vdash t\colon \sigma$. \\[.5em]
  \noindent $\Gamma \vdash u^{\rho \to \sigma}v^{\rho}$:
  Immediate from induction hypothesis.
\end{proof}

We now complete the proof of bar theorem.
\begin{proof}[Proof of Theorem \ref{thm:BarCont}]
  Let $Y \colon \Baire \to \Base$ be a closed term of $\HAw$.
  Then, $Y\Omega$ is a closed term of type $\Base$ in indeterminate
  $\Omega$. Thus,
  $\gamma \equiv (Y\Omega)^{\dagger}$ is a neighbourhood
  function of $Y$ by Lemma \ref{lem:RelTranslation}.
  By Lemma \ref{lem:BarProp}, we also have
  that
  $\HAw \vdash P^{Q}_{\Base}(\gamma)$ 
  for any predicate $Q$ on $\Base^{*}$.
  Hence
    $
     \HAw \vdash S_{\gamma} \subseteq Q \wedge \Ind(Q)\imp Q(\nil).
    $
\end{proof}

\section{Applications of bar theorem}\label{sec:Application}
We prove some well-known properties of $\HAw$
as applications of bar theorem.

\subsection{Uniform continuity on the Cantor space}\label{sec:UCont}
We show that the restriction of a closed term 
$Y \colon \Baire \to \Base$ to the Cantor space (the space of binary
sequences) is uniformly continuous.
\begin{theorem}
  For any closed term $Y \colon \Baire \to \Base$ of $\HAw$,
  \[
    \HAw \vdash \exists n^{\Base} \forall \alpha^{\Cantor}\forall
    \beta^{\Cantor}
    \left[ \overline{\alpha}n = \overline{\beta}n \imp Y\alpha = Y \beta \right].
  \]
\end{theorem}
  Here, we use the abbreviation $\forall \alpha^{\Cantor} A(\alpha)  \equiv 
  \forall
    \alpha^{\Baire} \left[ \forall n^{\Base}\alpha n  \leq 1 \imp A(\alpha)  \right].
    $
\begin{proof}
    Let $Y$ be a closed term of type $\Baire \to \Base$.  By Theorem
    \ref{thm:BarCont}, there is a neighbourhood function $\gamma
    \colon \Base^{*} \to \Base$ of $Y$  which satisfies \eqref{eq:BI}
    for any predicate $Q$ on $\Base^{*}$.
    In particular, consider a predicate $Q$ on $\Base^{*}$ defined as
      \[
        Q(a) \equiv \Binary(a) \imp 
        \exists n^{\Base}
        \forall \alpha^{\Cantor}
        \beta^{\Cantor}
        \left[ \overline{\alpha}n = \overline{\beta}n \imp
        Y(a* \alpha) = Y(a * \beta) \right],
      \]
      where $\Binary(a) \equiv \forall n < \lh{a} \;  a_{n} \leq 1$.
     Clearly, $S_{\gamma} \subseteq Q$ and $Q$ is inductive. Hence $Q(\nil)$,
     which is the statement to be proved.
\end{proof}

\subsection{Closure under the rule of bar induction}
\label{sec:BIRule}
The \emph{rule of bar induction} (without parameters) says that
for any predicate $P$ on $\Base^{*}$ without
parameters other than $a$, if
\[
  \HAw \vdash \forall \alpha^{\Baire} \exists n^{\Base}
  P(\overline{\alpha}n) \wedge \text{$P$ is monotone and decidable}
\]
then for any predicate $Q$ on $\Base^{*}$, 
\begin{equation}\label{eq:BarRuleConclusion}
  \HAw \vdash P \subseteq Q \wedge \Ind(Q) \imp Q(\nil).
\end{equation}
\begin{theorem} \label{thm:BarIndRule}
$\HAw$ is closed under the rule of bar induction.
\end{theorem}
\begin{proof}
  Let $P(a^{\Base^{*}})$ be a predicate on $\Base^{*}$ without
parameters other than $a$. Suppose that
\[
  \HAw \vdash \forall \alpha^{\Baire} \exists n^{\Base}
  P(\overline{\alpha}n) \wedge \text{$P$ is  monotone and decidable}.
\]
By the modified realizability \cite[Chapter 3, Section
4]{Troelstra1973}, we have a closed term $Y \colon \Baire \to \Base$ such that
\[
  \HAw \vdash \forall \alpha^{\Baire} P(\overline{\alpha}Y\alpha).
\]
By Theorem \ref{thm:BarCont}, there is a neighbourhood function $\gamma \colon
 \Base^{*} \to \Base$  of $Y$ which
 satisfies \eqref{eq:BI} for any predicate $Q$ on $\Base^{*}$.

 Let $Q$ be a predicate on $\Base^{*}$. We show that
 \begin{equation}\label{eq:PtoQ}
   P \subseteq Q \wedge \Ind(Q) \imp S_{\gamma} \subseteq Q,
 \end{equation}
 from which \eqref{eq:BarRuleConclusion} follows immediately.
Suppose that $P \subseteq Q$ and that $Q$ is inductive. Let
$a^{\Base^{*}}$ satisfy $S_{\gamma}$, i.e., $\gamma(a) > 0$, 
and put $k := \gamma(a) \dotminus \lh{a}$. Since $P$ is monotone, 
we have $P(a*b)$ for all $b$ such that $\lh{b} = k$.
Thus $Q(a*b)$ for all $b$ such that $\lh{b} = k$.
Hence $Q(a)$ by Lemma \ref{lem:inductive}. Therefore $S_{\gamma} \subseteq Q$.
\end{proof}

\subsection{Closure of bar recursion for the lowest type}\label{sec:BarRecursion}
For each pair of types $\tau, \sigma$, Spector's bar recursion
is the following schema:
\begin{align}\label{eq:BarRec}
  \BR^{\tau,\sigma}(Y,G,H)(a) = \begin{cases}
    G(a)& \text{if $Y(\AppendZero{a}) < \lh{a}$}\\
    H(a, \lambda x. \BR^{\tau,\sigma}(Y,G,H)(a * \langle x \rangle))& \text{otherwise}
\end{cases}
\end{align}
where $a \colon \tau^{*}$, $G \colon \tau^{*} \to \sigma$, $H \colon \tau^{*} \to
(\tau \to \sigma) \to \sigma$, and $Y \colon (\Base \to \tau) \to
\Base$.%
\footnote{Our definition of $\HAw$ does
not include type $\tau^{*}$ of finite sequences for an arbitrary type~$\tau$. 
Hence, bar recursion is understood to
be formulated in an extension of $\HAw$ with the type $\tau^{*}$ of
finite sequences for each type $\tau$. 
}
We call a function $\BR^{\tau,\sigma}$ of  type
\[
 ((\Base \to \tau) \to \Base)  \to (\tau^{*} \to \sigma) \to (\tau^{*} \to (\tau \to \sigma) \to \sigma) \to \tau^{*} \to \sigma
\]
which satisfies \eqref{eq:BarRec} a \emph{bar recursor} of types
$\tau$ and $\sigma$. The first argument of a bar recursor, i.e., a
function of type $(\Base \to \tau) \to \Base)$, is called a
\emph{stopping function} of bar recursion.

Schwichtenberg \cite{SchwichtenbergBarRec01} showed that if
$Y,G$, and $H$ are closed terms of G\"odel's system~$\SystemT$ and the
type $\tau$ is of level $0$ or $1$, then the function $\lambda a.
\BR(Y,G,H)(a)$ which satisfies \eqref{eq:BarRec} is
$\SystemT$-definable. His proof requires a detour through a
system based on infinite terms. Oliva and Steila
\cite{Oliva_Steila_bar_recursion_closure} strengthened
Schwichtenberg's result by giving an explicit construction of the
function $\lambda G.\lambda H.\lambda a. \BR(Y,G,H)(a)$ from a closed
term $Y \colon (\Base \to \tau) \to \Base$ (for type $\tau$ of level $0$
and $1$), and showed that it satisfies the defining equation of bar
recursion for any $G, H$ and $a$.

We give another construction of a bar recursive function for the
lowest type from a closed term $Y \colon \Baire \to \Base$ using
modified realizability.

\begin{theorem}\label{thm:ClosureBarRec}
For any type $\sigma$ and a closed term $Y \colon \Baire \to \Base$,
there exists a closed term $\xi$ of type 
\[
  (\Base^{*} \to \sigma) \to (\Base^{*} \to (\Base \to \sigma) \to \sigma) \to \Base^{*} \to \sigma
\]
which satisfies the defining equation of the bar recursion for
$Y$, i.e., 
\begin{equation}\label{eq:BarRecforY}
  \HAw \vdash
  \forall G \forall H \forall a \;
  \xi(G,H)(a) = \begin{cases}
    G(a)& \text{if $Y(\AppendZero{a}) < \lh{a}$}\\
    H(a, \lambda x. \xi(G,H)(a * \langle x \rangle))&
    \text{otherwise}.
\end{cases}
\end{equation}
\end{theorem}
\begin{proof}
    Fix a type $\sigma$ and a closed term $Y \colon \Baire  \to
    \Base$.
    By Theorem \ref{thm:BarCont}, there is a closed term
    $\gamma:{\Base^{*} \to \Base}$ which is a neighbourhood function
    of $Y$ and satisfies \eqref{eq:BI} for any predicate $Q$ on
    $\Base^{*}$.
    Define a predicate $P_{Y}$ on $\Base^{*}$ by  
    \[
      P_{Y}(a) \equiv Y(\AppendZero{a}) < \lh{a}.
    \]
    For any predicate $Q$ on $\Base^{*}$, it is straightforward to
    show that 
    \[
      P_{Y} \cap S_{\gamma} \subseteq Q \wedge \Ind(Q) \imp S_{\gamma}
      \subseteq Q.
    \]
    See the argument following \eqref{eq:PtoQ}. Thus
    \begin{equation*}
      \HAw \vdash P_{Y} \cap S_{\gamma} \subseteq Q \wedge \Ind(Q) \imp Q(\nil).
    \end{equation*}
    Define a predicate $Q$ on ${\Base^{*}}$ by  
    \[
      Q(a) \equiv \exists \xi \; \calBR(\xi,a),
    \]
  where $\xi$ is of type
  $
  (\Base^{*} \to \sigma) \to (\Base^{*} \to (\Base \to \sigma) \to
  \sigma) \to \Base^{*} \to \sigma
  $ 
  and  $\calBR(\xi,a)$ is the following formula:
  \begin{align*}
    \forall G \forall H \forall b \; \xi(G,H)(b) = \begin{cases}
      G(b)& \text{if $Y(\AppendZero{a*b}) < \lh{a} + \lh{b}$}\\
      H(b, \lambda x. \xi(G,H)(b * \langle x \rangle))&
      \text{otherwise}.
   \end{cases}
  \end{align*}
  We show that $P_{Y} \cap S_{\gamma}
  \subseteq Q$ and that $Q$ is inductive.  
  \smallskip

  \noindent \emph{$P_{Y} \cap S_{\gamma} \subseteq Q$}: Let $a^{\Base^{*}}$ 
  satisfy $P_{Y} \cap S_{\gamma}$. Then, for any
  $b^{\Base^{*}}$ we have 
  $Y(\AppendZero{a*b}) = Y(\AppendZero{a}) < \lh{a} \leq \lh{a} + \lh{b}$. 
  Hence, the function $\xi := \lambda G. \lambda H. G$ witnesses
  $Q(a)$.
  \smallskip

  \noindent\emph{$Q$ is inductive}:
  Suppose that $\forall x^{\Base} Q(a * \seq{x})$.
  For each $x^{\Base}$, there exists a function $\xi_{x}$ such that
  $\calBR(\xi_{x},a * \seq{x})$.
  By countable choice ($\CC$), there exists a sequence
  $(\xi_{x})_{x^{\Base}}$ of functions  such that
  $\calBR (\xi_{x}, a*\seq{x})$ for each $x$.

  For any $G \colon \Base^{*} \to \sigma$, $H\colon \Base^{*} \to (\Base \to
  \sigma) \to \sigma$ and $b\colon\Base^{*}$, define $\xi(G,H)(b)$ by
  induction on the length of $b$:
  \begin{align*}
    \xi(G,H)(\nil) &:=
    \begin{cases}
      G(\nil) & \text{if $Y(\AppendZero{a}) < \lh{a}$}\\
      H(\nil, \lambda x. \xi_{x}(G_{x}, H_{x})(\nil))
      &\text{otherwise}
    \end{cases}\\
    \xi(G,H)(\seq{x}* b) &:=
    \xi_{x}(G_{x},H_{x})(b)
  \end{align*}
  where $G_{x}$ and $H_{x}$ are defined by
  \begin{align*}
    G_x(b) &:= G(\seq{x} * b),
    \\
    H_x(b,f) &:= H(\seq{x}*b,f).
  \end{align*}
  We must show that $\calBR(\xi, a)$. Fix $G,H$, and  $b$. We distinguish
  two cases:
  \smallskip

  \noindent\emph{Case $b \equiv \nil$}: If $Y(\AppendZero{a}) < \lh{a}$, then
  $\xi(G,H)(\nil) = G(\nil)$. Otherwise
  \[
    \xi(G,H)(\nil) = H(\nil, \lambda x. \xi_{x}(G_{x},H_{x})(\nil)) =
H(\nil, \lambda x. \xi(G,H)(\seq{x})).
\]
  \noindent\emph{Case $b \equiv \seq{x} * b'$ for some $b'$}: 
  If $Y(\AppendZero{a * b}) < \lh{a} + \lh{b}$, then 
  $Y(\AppendZero{a * \seq{x}* b'}) < \lh{a * \seq{x}} + \lh{b'}$. Thus
    $
    \xi(G,H)(b) = \xi_{x}(G_{x}, H_{x})(b') = G_{x}(b') = G(b).
    $
  Otherwise
  \begin{align*}
    \xi(G,H)(b) &= \xi_{x}(G_{x},H_{x})(b') \\
    &= H_{x}(b', \lambda y.  \xi_{x}(G_{x},H_{x})(b' * \seq{y})) \\
    &= H_{x}(b', \lambda y. \xi(G,H)(b * \seq{y}))\\
    &= H(b, \lambda y. \xi(G,H)(b * \seq{y})).
  \end{align*}
  Hence $\calBR(\xi, a)$. Thus $Q$ is inductive.
  \smallskip

  Therefore $Q(\nil)$, and so
  \[
    \HAw + \CC \vdash \exists \xi \calBR(\xi,\nil).
  \]
  Since $\calBR(\xi,\nil)$ is a purely universal statement, 
  modified realizability yields a witness $\xi$ as a
  closed term of $\HAw$ for which \eqref{eq:BarRecforY} holds.
  Note that since countable choice is modified realizable, the use of
  countable choice is eliminated in the last step.
\end{proof}

\subsection*{Acknowledgements}
I thank Mart{\'\i}n Escard\'o, Makoto Fujiwara, Paul-Andr\'{e}
Melli\`{e}s, Paulo Oliva, and Giuseppe Rosolini for useful
discussions. I also thank Ulrich Kohlenbach for the reference to his
proof of the fan rule for $\HAw$.  This work was carried out while I
was visiting the Hausdorff Research Institute for Mathematics (HIM),
University of Bonn, for their trimester program ``Types, Sets and
Constructions'' (May--August 2018). I thank the institute for their
support and hospitality, and the organisers of the program for
creating a stimulating environment for research.


\end{document}